\documentclass[leqno,a4paper,intlimits]{amsart}

\usepackage{amsmath,amssymb,amsthm}
\usepackage{enumerate}
\usepackage{graphics,graphicx,color}
\usepackage{multicol}
\usepackage{mathrsfs}

\newtheorem{theorem}{Theorem}[section]

\newtheorem{lemma}[theorem]{Lemma}
\newtheorem{proposition}[theorem]{Proposition}
\theoremstyle{definition}

\newtheorem{assumption}[theorem]{Assumption}

\newtheorem{remark}[theorem]{Remark}

\def\ran{\mathop{\mathrm{ran}}}
\def\dd{\mathrm{d}}
\def\dsigma{\:\dd\sigma}
\def\ee{\mathrm{e}}

\def\RR{{\mathbb{R}}}%
\def\NN{{\mathbb{N}}}

\def\calT{{\mathcal{T}}}

\def\calA{{\mathcal{A}}}

\def\calP{{\mathcal{P}}}
\def\calO{{\mathcal{O}}}
\def\calI{{\mathcal{I}}}
\def\calE{{\mathcal{E}}}
\def\calG{{\mathcal{G}}}

\def\calT{\mathcal{T}}

\def\calE{\mathcal{E}}
\def\calD{\mathcal{D}}

\def\W{\mathrm{W}}
\def\Ell{\mathrm{L}}
\def\BV{\mathrm{BV}}
\def\Lip{\mathrm{Lip}}
\def\Ce{\mathrm{C}}

\def\LLL{\mathscr{L}}

\def\dx{\xi}

\begin{document}

 \title[Splitting for dissipative delay]{Operator splitting for dissipative delay equations}

\author[A. B\'{a}tkai]{Andr\'{a}s B\'{a}tkai}
\address{A.B., E\"{o}tv\"{o}s Lor\'{a}nd University, Institute of Mathematics,
  1117 Budapest, P\'{a}zm\'{a}ny P\'eter s\'{e}t\'{a}ny 1/C, Budapest, Hungary}
\curraddr{Bergische Universit\"at Wuppertal, School of Mathematics and Natural Sciences, Gau\ss
  strasse 20, 42119, Wuppertal, Germany}
\email{batka@cs.elte.hu}

\author[P. Csom\'{o}s]{Petra Csom\'{o}s}
\address{P.Cs., E\"{o}tv\"{o}s Lor\'{a}nd University, Institute of Mathematics,
  and MTA-ELTE Numerical Analysis and Large Networks Research Group, Hungarian Academy of Sciences,
1117 Budapest, P\'{a}zm\'{a}ny P\'eter s\'{e}t\'{a}ny 1/C, Budapest, Hungary}
\email{csomos@cs.elte.hu}

\author[B. Farkas]{B\'{a}lint Farkas}
\address{B.F., Bergische Universit\"at Wuppertal, School of Mathematics and Natural Sciences, Gau\ss
  strasse 20, 42119, Wuppertal, Germany}
\email{farkas@uni-wuppertal.de}



\date\today
\begin{abstract}
We investigate Lie--Trotter product formulae for abstract nonlinear evolution equations with delay. Using results from the theory of nonlinear contraction semigroups in Hilbert spaces, we explain the convergence of the splitting procedure. The order of convergence is also investigated in detail, and some numerical illustrations are presented.
\end{abstract}
\maketitle
\section{Introduction}

From its very beginning, operator semigroup theory has played an important role in explaining the convergence of numerical methods, see the seminal paper by Lax and Richtmyer \cite{Lax-Richtmyer}, or the monograph by Butzer and Behrens \cite{Butzer-Berens}.

The Lie--Trotter product formula and its generalizations to nonlinear semigroups play a central role in semigroup theory, as witnessed in the monographs by Goldstein \cite{Goldstein_book} or by Engel and Nagel \cite{Engel-Nagel}. And this product formula is widely used in numerical analysis under the name of operator splitting.

Our aim is to investigate the convergence of the Lie--Trotter product formula (or operator splitting) for nonlinear partial differential equations with delay.

Such equations play an important role in modeling physical, chemical, economical, etc.~phenomena, because it is quite natural to assume that past occurrences effect the model. For further motivation see for example the monographs by Wu \cite{Wu} or B\'{a}tkai and Piazzera \cite{Batkai-Piazzera}.

There has been lots of work describing the asymptotic behavior and regularity of solutions, as well as in the numerical analysis of ordinary differential equations with delay, see for example the monograph by Bellen and Zennaro \cite{Bellen-Zennaro}. The numerical analysis of partial differential equations with delay, however, seems to be in its infancy. There are only sporadic works on this subject, like the interesting results of Jackiewicz, Mead, and  Zubik--Kowal \cite{jz,mz} on the pseudospectral method. The present paper aims to contribute to this topic by analyzing an operator splitting procedure for nonlinear partial differential equations with delay.

The idea of operator splitting is to decompose the differential equation into simpler equations which can be solved in an effective way, and then to represent the solution of the original equation using product formulae, like the Lie--Trotter one. For ordinary
differential equations, the theory seems to be quite complete, as witnessed in
Hairer et al.~\cite[Section II.4,5]{Hairer-Lubich-Wanner}. There has
also been enormous progress in the theoretical investigation of splitting
procedures for infinite dimensional systems in recent years, see for example
the monographs by Farag\'{o} and Havasi \cite{Farago-Havasi_book}, Holden et
al.~\cite{Holden-Karlsen-Lie-Risebro}, or Lubich \cite{Lubich08}. See also the recent
papers by B\'{a}tkai et
al.~\cite{Batkai-Csomos-Farkas-Nickel1,Batkai-Csomos-Farkas-Nickel2,Batkai-Csomos-Nickel}, where
nonautonomous equations and spatial approximations are also
considered.  Unfortunately, abstract results analyzing the order of convergence
are rather incomplete, and, as it seems, can be applied to delay equations
only with considerable difficulty.

The idea to apply splitting procedures to delay equations is the following. Consider, e.g., a delayed reaction--diffusion equation of the form

\begin{equation*}
u'(t) = \Delta u(t) + g\left(u(t-1)+\int\nolimits_{-1}^{0} \eta(\sigma)\cdot u(t+\sigma) \dsigma\right).
\end{equation*}
with some initial and boundary conditions. The delay term appearing here represents the two main classes of possible delays in applications: point delays corresponding to dependence on a single event in the past, and distributed delays (given by an integral term with an integrable kernel $\eta$) corresponding to dependence on a whole time period in the past. In our opinion, distributed delays are often more realistic in modeling, but we will not restrict ourselves to them here.

The nonlinearity $g$ can be chosen from a wide range of functions depending on applications, for example a rational function in chemical reactions, or a $\tanh$ type function in neural networks, see for example the monograph by Wu \cite{Wu}.

Since the delay term is in a way a ``scalar operator'' (in the sense that it does not mix the spatial variables), it is natural to decompose the equation into two sub-problems: the heat equation
\begin{equation*}
w'(t) = \Delta w(t),
\end{equation*}
and a scalar-valued delay equation
\begin{equation*}
v'(t) =  g\left(v(t-1)+\int\nolimits_{-1}^{0} \eta(\sigma)\cdot v(t+\sigma) \dsigma\right).
\end{equation*}
Both equations can be solved numerically in an effective way. Note that the second equation becomes here an ordinary differential equation with delay, hence the methods described in Bellen and Zennaro \cite{Bellen-Zennaro} can be applied.

Our aim is to put this example in an abstract theoretical perspective explaining the convergence  and analyze the order of convergence in some special cases. To be more specific, we will consider on the Hilbert space $H$ the following abstract delay differential equation
\begin{equation*}
\left\{
\begin{aligned}
\frac{\dd u(t)}{\dd t}&=Bu(t)+\Phi u_t, \qquad t\ge 0, \\
u(0)&=x\in H, \\
u_0&=f\in\Ell^p\big([-1,0];H\big),
\end{aligned}
\right.
\end{equation*}
where $u:[-1,\infty)\to H$ is unknown and the \emph{history function} $u_t$ defined by $u_t(\sigma):=u(t+\sigma)$ for $\sigma\in[-1,0]$. The further precise definitions and assumptions will be made in Section \ref{sec:delay_sg}, see also the monograph B\'{a}tkai and Piazzera \cite{Batkai-Piazzera}.
In many cases, as explained above, it is easier to solve the equation without delay and the ``pure'' delay equation separately. In this case, it is natural to apply some operator splitting procedure described below.

Let us fix a time step $h >0$ and  solve first the equation
\begin{equation}\label{eq:spl1}
\left\{
\begin{aligned}
\frac{\dd v^{(1)}(t)}{\dd t}&=\Phi v^{(1)}_t, \qquad t\in [0,h], \\
v^{(1)}(0)&=x=:x_1 \\
v^{(1)}_0&=f=:f_1.
\end{aligned}
\right. \tag{SPL/1}
\end{equation}
Then for $y_1:= v^{(1)}(h)$ we solve the equation
\begin{equation}\label{eq:spl2}
\left\{
\begin{aligned}
\frac{\dd w^{(1)}(t)}{\dd t}&= Bw^{(1)}(t), \qquad t\in [0,h], \\
w^{(1)}(0)&=y_1.
\end{aligned}
\right. \tag{SPL/2}
\end{equation}
To initialize the next step we set
\begin{align*}
x_2&:= w^{(1)}(h)\\
f_2&:=v^{(1)}_h.
\end{align*}
Note that (SP1/2) is an undelayed equation, hence the history function remains unchanged, when applying this solution step. This is why we took the initialization above for $f_2$.
 We iterate this procedure and the \emph{sequentially split solution} at time level $t=kh$ is then
\begin{equation*}
u^{\text{sq}}(kh):=x_{k+1}.
\end{equation*}
We shall show that for fixed $t\in [0,\infty)$ and for $h\to 0$ ($t=kh$, so $k\to\infty$) this split solution $u^{\text{sq}}(kh)$ converges to $u(t)$.

\medskip This procedure is especially useful, if we can drastically reduce the computational complexity of the problem by the splitting. This is, e.g., the case if $\Phi=C\delta(-1)$ is a point delay, see also Section \ref{sec:num}. Hence, we can integrate the first split equation explicitly, reducing the problem to solving the second equation, a classical partial differential equation. For a different splitting procedure, designed specifically for distributed delays, we refer to Csom\'{o}s and Nickel \cite{Csomos-Nickel} and B\'{a}tkai et al.~\cite{Batkai-Csomos-Farkas}.

\medskip Besides the sequential splitting there are other splitting methods known in the literature. However for the sake of simplicity, in this paper we only deal  with  the sequential and the first order Lie splitting. The convergence of higher order splitting schemes can be explained along the same lines.

\medskip  In Section \ref{sec:delay_sg}, we show a way to rewrite a delay
equation as an abstract Cauchy problem, and show how to associate nonlinear
contraction semigroups to that. Also basic facts and results  about such
semigroups are recalled there. These are then used to explain the convergence
of the splitting procedure. In Section \ref{sec:order_conv} we study the order
of convergence in an important special case, when the undelayed part is
linear, and the delay term can be nonlinear. The last Section  \ref{sec:num}
is devoted to numerical experiments, where we numerically illustrate the
theoretical results about the order of convergence, and compare the proposed splitting
method to the implicit Euler method.


\section{The delay semigroup}\label{sec:delay_sg}

\noindent Consider the \emph{abstract delay equation} of the following form (see, e.g., B\'{a}tkai and Piazzera \cite{Batkai-Piazzera}):
\begin{equation}
\left\{
\begin{aligned}
\frac{\dd u(t)}{\dd t}&=Bu(t)+\Phi u_t, \qquad t\ge 0, \\
u(0)&=x\in H, \\
u_0&=f\in\Ell^p\big([-1,0];H\big)
\end{aligned}
\right.
\label{delay}
\end{equation}
on the Hilbert space $H$ where $1\leq p<\infty$ and $u_t$ is the history function, i.e.
$$
u_t(\sigma)=u(t+\sigma),\quad t\geq 0, \:\sigma\in [-1,0],
$$
where we used the initial condition given in \eqref{delay} as $u(s)=f(s)$ for $s\in[-1,0]$.

It is possible to transform the delay equation \eqref{delay} into an abstract Cauchy problem, see B\'{a}tkai and Piazzera \cite{Batkai-Piazzera}. In order to do so, we take the product space $\mathcal{E}:=H\times \Ell^p\big([-1,0];H\big)$ and the new unknown function as
\begin{equation}
t\mapsto\mathcal{U}(t):=\binom{u(t)}{u_t}\in\mathcal{E}.
\nonumber
\end{equation}
Then \eqref{delay} can be written as an abstract Cauchy problem on the space $\mathcal{E}$
\begin{equation}
\left\{
\begin{aligned}
\frac{\dd\mathcal{U}(t)}{\dd t}&=\mathcal{G}\mathcal{U}(t), \qquad t\ge 0, \\
\mathcal{U}(0)&=\tbinom{x}{f}\in\mathcal{E},
\end{aligned}
\right.
\label{acp_delay}
\end{equation}
where the operator $\mathcal{G}$ is given by the operator matrix
\begin{equation}
\mathcal{G}:=\left(\begin{array}{cc} B & \Phi \\ 0 & \frac{\dd}{\dsigma} \end{array}\right)
\end{equation}
on the domain
\begin{equation}
D(\mathcal{G}):=\big\{\tbinom{x}{f}\in D(B)\times \W^{1,p}\big([-1,0];H\big): \ f(0)=x  \big\}.
\nonumber
\end{equation}
\noindent It is shown in B\'{a}tkai and Piazzera \cite[Corollary 3.5, Proposition 3.9]{Batkai-Piazzera} that the delay equation \eqref{delay} and the abstract Cauchy problem \eqref{acp_delay} are equivalent, i.e., they have the same solutions. More precisely, the first coordinate of the solution of \eqref{acp_delay} always solves \eqref{delay}.

\medskip Now the abstract Cauchy problem \eqref{acp_delay} can be solved by semigroup methods. These we recall briefly.
The best Lipschitz constant of a Lipschitz continuous function $T:X\to X$ on a Banach space $X$ is denoted by $\|T\|_{\Lip}$. The operator $T$ is called then a \textit{contraction} if $\|T\|_{\Lip}\leq 1$. For further notation, terminology and results on m-dissipative operators and nonlinear contraction semigroups, we refer to the monograph by Miyadera \cite{miyadera:77}.

\begin{assumption}\label{ass:1} Suppose that
\begin{enumerate}
\item there is $\alpha\in\RR$ such that $B-\alpha I$ is m-dissipative and hence $B$ is the generator of a strongly continuous (nonlinear) semigroup $V$ on $H$, and
\item $\Phi:\W^{1,p}\big([-1,0];H\big) \to H$ is an operator defined as follows. Let $g:H\to H$ be a Lipschitz continuous function with Lipschitz constant $\beta$, and let $\eta:[-1,0]\to \LLL(H)$ be a function of bounded variation with values as bounded linear operators. Suppose further that
    \begin{equation*}
    \eta(-1)=0 \text{ and that } \lim_{\sigma\to -1} \eta(\sigma)\neq 0.
    \end{equation*}
    Now the operator $\Phi$ is given as $\Phi=g\circ \Psi$, where
    \begin{equation*}
    \Psi(f):= \int\nolimits_{-1}^0 \dd\eta(\sigma) f(\sigma)\qquad\qquad \text{for } f\in \Ce([-1,0];H).
    \end{equation*}
\end{enumerate}
\end{assumption}
  Following Webb \cite[Section 4]{Webb76}, it is possible to find a new equivalent norm in $\calE$ such that the operator $\calG-\gamma\calI$ becomes an m-dissipative operator and hence the generator of a (nonlinear) semigroup. We define
\begin{equation*}
\tau(r):=\int\nolimits_{-1}^r \|\dd\eta(\sigma)\|
\end{equation*}
the total variation of $\eta$ on $[-1,r]$, and introduce the new norm, equivalent to the original one, as
\begin{equation*}
\bigl\|\tbinom{x}{f}\bigr\|_\tau := \bigr( \|x\|^p + \int\nolimits_{-1}^0 \|f(\sigma)\|^p\tau(\sigma)\dsigma \bigr)^{{1}/{p}}\quad \mbox{for $\tbinom{x}{f}\in \calE$}.
\end{equation*}
It follows from the construction that we can also identify the constants as
\begin{equation}\label{eq:const_eq_norm}
\min\{1,\tau(-1)\}\bigl\|\tbinom{x}{f}\bigr\|^p\leq \bigl\|\tbinom{x}{f}\bigr\|^p_\tau\leq \max\{1,\tau(0)\}\bigl\|\tbinom{x}{f}\bigr\|^p.
\end{equation}

We shall also use the notation $\Ell^p([-1,0],\tau;H)$ and $\W^{1,p}([-1,0],\tau;H)$ if we want to stress that we consider the spaces $\Ell^p([-1,0];H)$ and $\W^{1,p}([-1,0];H)$ with the equivalent norms defined by
$$
\|f\|_{\tau,p}^p = \int\nolimits_{-1}^0 \|f(\sigma)\|^p\tau(\sigma)\dsigma
$$
and
$$
\|f\|_{\tau,1,p}^p = \int\nolimits_{-1}^0 \|f(\sigma)\|^p\tau(\sigma)\dsigma + \int\nolimits_{-1}^0 \|f'(\sigma)\|^p\tau(\sigma)\dsigma,
$$
respectively.

\medskip The abstract Cauchy problem \eqref{acp_delay} can be solved by semigroup theoretic methods. These we recall briefly.

\medskip The best Lipschitz constant of a Lipschitz continuous function $T:X\to X$ on a Banach space $X$ is denoted by $\|T\|_{\Lip}$. The operator $T$ is called then a \textit{contraction} if $\|T\|_{\Lip}\leq 1$. A subset $A\subset X\times X$ is called \textit{dissipative} if for each $\alpha>0$, $(I-\alpha A)^{-1}$ is a function and a contraction. We call $A$ \textit{m-dissipative} if $\overline{D(A)}=X$ and for all $\alpha>0$ $\ran(I-\alpha A) = X$ ($D(A)$ is the domain of definition of $A$), while for $\omega\in\RR$ the set $A$ is called \textit{$\omega$-m-dissipative} if $A-\omega I$ is m-dissipative.
A family $T$ of Lipschitz continuous operators on the Banach space $X$ is called a \emph{strongly continuous semigroup} of type $\omega\in\RR$ (or a quasi-contraction semigroup) if
\begin{enumerate}
\item $T(t):X\to X$ for all $t\geq 0$,
\item $T(t)T(s)x = T(t+s)x$ for all $t,s\geq 0$, and  $T(0)=I$,
\item $\|T(t)\|_{\Lip}\leq \ee^{\omega t}$ for all $t\geq 0$,
\item $\lim\limits_{t\to 0} T(t)x = x$.
\end{enumerate}
The central point in this theory is the theorem of Crandall and Liggett \cite{crandall-liggett} generalizing the classical Hille--Yosida Theorem.
\begin{theorem}
Let $A\subset X\times X$ be $\omega$-m-dissipative for some $\omega\in\RR$.  Then $A$ generates a strongly continuous semigroup of type $\omega$ given by the formula
\begin{equation}
T(t)x = \lim_{n\to\infty}\left(I-\tfrac{t}{n}A\right)^{-n}x
\end{equation}
for all $x\in X$ and $t\geq 0$.
\end{theorem}
If $A$ generates a semigroup in this way, we call the abstract Cauchy problem associated to $A$ well-posed on $X$.

Then we have the following result, see Webb \cite[Proposition 4.1]{Webb76}.
\begin{theorem}\label{thm:webb}
If $p>1$, then $\calG-\gamma\calI$ is m-dissipative in $\calE$ for
\begin{equation*}
\gamma=\max\bigl\{0,\tau(0)\bigl(\tfrac{1}{p}+\tfrac{\beta^p}{q}\bigr)+\alpha\bigr\},
\end{equation*}
where $\tfrac{1}{p}+\tfrac{1}{q} = 1$. If $p=1$ and $\beta\leq 1$, then $\calG-\gamma\calI$ is m-dissipative in $\calE$, where $\gamma=\max\{0,\tau(0)+\alpha\}$.
\end{theorem}

Let us turn our attention to the splitting procedure described in the introduction.
On the semigroup level in the product space $\calE$ this corresponds to the splitting $\calG=\calA_1+\calA_2$, where
\begin{equation*}
\calA_1 = \begin{pmatrix} B & 0 \\ 0 & 0 \end{pmatrix}, \qquad \calA_2 = \begin{pmatrix} 0 & \Phi \\ 0 & \tfrac{\dd}{\dsigma} \end{pmatrix},
\end{equation*}
with
\begin{equation*}
D(\calA_1)=D(B)\times \Ell^p([-1,0],\tau;H)
\end{equation*}
and
\begin{equation*}
D(\calA_2)=\left\{ \tbinom{x}{f}\in H\times \W^{1,p}([-1,0],\tau;H)\,:\, f(0)=x\right\}.
\end{equation*}
These operators generate strongly continuous semigroups of the form
\begin{equation}\label{eq:split_sgs}
\calT_1(t) = \begin{pmatrix} V(t) & 0 \\ 0 & \widetilde{I} \end{pmatrix}, \qquad \calT_2(t)\binom{x}{f} = \begin{pmatrix} v(t) \\ v_t \end{pmatrix},
\end{equation}
with $\widetilde{I}$ being the identity operator on $\Ell^p([-1,0],\tau;H)$ and
\begin{equation*}
v(t)= x+  \int\nolimits_0^t \Phi v_{r} \dd r,
\end{equation*}
see Webb \cite[Proposition 5.12]{Webb76}. If we denote the projection to the first coordinate by
\begin{equation*}
\mathcal{P}_1\tbinom{x}{f}:=x,
\end{equation*}
then the sequential splitting (after $k$ steps with time step $h$) can be written as
\begin{equation*}
u^{\text{sq}}(hk)=\mathcal{P}_1\left[\left(\calT_1(h)\calT_2(h)\right)^k \tbinom{x}{f}\right],
\end{equation*}
and the Lie splitting analogously as
\begin{equation*}
u^{\text{Lie}}(hk)=\mathcal{P}_1\left[\left((I-h\calA_1)^{-1}(I-h\calA_2)^{-1}\right)^k \tbinom{x}{f}\right]
\end{equation*}
for all $h>0$ sufficiently small.

Our first aim is to show the convergence
$u^{\text{spl}}(hk)\to u(t)=\calP_1 \mathcal{U}(t)$ for $k\to \infty$ and $kh=t$, where $\text{spl}$ stands for $\text{sq}$ or $\text{Lie}$, respectively.

\medskip
The main abstract technical tool in investigating splitting procedures are Lax--Chernoff type theorems and variants of the Lie--Trotter product formula for nonlinear semigroups. Such result were proved in the paper by Brezis and Pazy \cite[Section 3]{Brezis-Pazy}, and improved by Kobayashi \cite{Kobayashi1}. We recall here the  results important for our investigation.
\begin{theorem}\label{thm:BrPa1}
\begin{enumerate}[a)]
\item Suppose that $A_1$, $A_2$, and $\overline{A_1+A_2}$ are $\omega$-m-dissipative sets in a Banach space $X$ generating the semigroups $T_1$, $T_2$, and $T$, respectively. Then
\begin{equation*}
U(t)x=\lim_{n\to\infty} \left[\left(I-\tfrac{t}{n}A_1\right)^{-1}\left(I-\tfrac{t}{n}A_2\right)^{-1}\right]^n x
\end{equation*}
for all $x\in X$.
\item If, in addition, $X=H$ is a Hilbert space and $A_1+A_2$ is already closed, then
\begin{equation*}
T(t)x=\lim_{n\to\infty} \left[T_1\left(\tfrac{t}{n}\right) T_2\left(\tfrac{t}{n}\right)\right]^n x.
\end{equation*}
\end{enumerate}
\end{theorem}
We shall refer to the case a) as the \emph{Lie-splitting}, and to case b) as the \emph{sequential splitting}. For more general product formulae, see the seminal paper by Lions and Mercier \cite{Lions-Mercier}.

\medskip The extra condition about the Banach space in b) is important, as a counterexample by Kurtz and Pierre \cite{Kurtz-Pierre} shows. Nevertheless, one can fairly relax this condition by requiring that $X$ has uniformly G\^ateaux differentiable norm, see Kobayashi \cite{Kobayashi2}. It seems to be folklore that a separable Banach space can be equivalently renormed such that the new norm has this differentiability property.

\medskip After all these preparations, we obtain the following general convergence result.
\begin{theorem}\label{thm:abstr_spl_delyay}
Suppose that Assumption \ref{ass:1} holds. Consider the delay equation \eqref{delay} and the splitting of the operator $\calG=\calA_1+\calA_2$ described above.
For every $p\in (1,\infty)$ the sequential and the Lie splittings converge in $H\times \Ell^p([-1,0],\tau;H)$.
\end{theorem}

\begin{proof}
By Theorem \ref{thm:webb} the semigroups $\calT_1$, $\calT_2$ and $\calT $ are all of type $\gamma$ for some $\gamma\in \RR$.
Note that for a Hilbert space $H$ and $1<p<\infty$, the norm of the $H$-valued Bochner space  $\Ell^p(H)$ is uniformly G\^ateaux differentiable (by a result of Day \cite{day1} the space $\Ell^p(H)$ is uniformly convex and uniformly smooth, which is even more than what is required; see also Diestel \cite[Sec.2.3]{diestel}). All in all we note that the above Theorem \ref{thm:BrPa1} is applicable in the situation of the delay equation, i.e., for the semigroups $\calT_1$ and $\calT_2$ on the product space $\calE=H\times \Ell^p([-1,0],\tau;H)$,  as described above. Notice that one has to endow the product space with an appropriate smooth and  uniformly convex product norm (see Clarkson \cite{Clarkson}).
Since $\calA_1+\calA_2$ generates $\calT$, Theorem \ref{thm:BrPa1} and the paragraph thereafter yield the proof.
\end{proof}

\section{Order of convergence}\label{sec:order_conv}

The convergence result we obtained in the previous section is rather general. Unfortunately, it is usually not possible to obtain estimates on the error of convergence in this generality. Hence, we turn our attention to some important special cases and prove convergence of the sequential splitting together with error estimates. In this section we suppose that operator $B$ generates a linear contraction semigroup. Further analysis of more complicated equations is subject to ongoing research.

\medskip
First we show how an abstract semigroup result can be applied directly to the problem, demonstrating the power of the semigroup approach.
To obtain some error estimates in the linear case, i.e., when $g(x)=x$, the identity, one can apply a general result by Hansen and Ostermann \cite[Theorem 3.1]{Hansen-Ostermann}.

\begin{theorem}\label{thm:previous}
Assume that $B$ generates a linear contraction $C_0$-semigroup and that the linear operator $\Phi$ satisfies the condition $\ran \Phi \subset D(B)$. Then for $\tbinom{x}{f}\in D(\calG^2)$ the sequential splitting {is of first order}.
\end{theorem}
Here the symbol $\ran \Phi$ denotes the range of the operator $\Phi$.
\begin{proof}
To prove the desired estimate, we have to check the conditions in Hansen and Ostermann \cite[Theorem 3.1]{Hansen-Ostermann} for the operators
\begin{equation*}
\calA_1 = \begin{pmatrix} B & 0 \\ 0 & 0 \end{pmatrix}, \qquad \calA_2 = \begin{pmatrix} 0 & \Phi \\ 0 & \tfrac{\dd}{\dsigma} \end{pmatrix},
\end{equation*}
with their respective domains, see Section \ref{sec:delay_sg}.

By Theorem \ref{thm:webb}, we know that each operator $\calA_i$ generates a quasi-contraction semigroup in $\calE$. Hence, we only have to check the domain condition
\begin{equation*}
D(\calG^2)\subset D(\calA_1\calA_2).
\end{equation*}

Consider first the right-hand side of this inclusion. Then
\begin{align*}
D(\calA_1\calA_2) &= \big\{\tbinom{x}{f}\in D(\calA_2)\,:\, \calA_2\tbinom{x}{f}\in D(\calA_1)\big\}\\
&=\big\{\tbinom{x}{f}\in D(\calA_2)\,:\, \Phi f \in D(B)\big\} = D(\calA_2)
\end{align*}
since $\ran\Phi\subset D(B)$.

For the left-hand side we obtain
\begin{multline*}
D(\calG^2) = \big\{\tbinom{x}{f}\in D(\calG)\,:\, \calG\tbinom{x}{f}\in D(\calG)\big\} = \big\{ \tbinom{x}{f}\in D(\calG)\,:\,\tbinom{Bx+\Phi f}{f'}\in D(\calG)\big\} = \\ \big\{\tbinom{x}{f}\in D(B)\times \W^{2,p}([-1,0],\tau;H)\,:\, f(0)=x, \, f'(0) = Bx+\Phi f\in D(B) \big\}.
\end{multline*}

\noindent Hence,
\begin{equation*}
D(\calG^2)\subset D(\calG) \subset D(\calA_2) = D(\calA_1\calA_2),
\end{equation*}
and the assertion is proved.
\end{proof}
We remark that Hansen and Stillfjord in their paper \cite{Hansen-Stillfjord} investigated Lie splitting for abstract nonlinear delay equations and obtained analogous results to ours.

\medskip Now we want to weaken the assumption about $\ran\Phi$ and still obtain a convergence rate result. To do so we directly calculate the local error of the splitting. The splitting procedure is certainly stable, as all appearing semigroups are of type $\gamma$.  After renorming the space $\calE$ as described in the paragraph preceding Theorem \ref{thm:webb} both  $\calT_1$ and $\calT_2$ become quasi-contraction semigroups of type $\gamma$, and thus in the new norm the splitting is stable. Hence (for any equivalent norm) there is a constant $\widetilde M\geq 1$ with
\begin{equation}\label{eq:stab}
\|(\calT_1(t)\calT_2(t))^k\|_{\Lip}\leq \widetilde M\ee^{2\gamma kt}\quad\mbox{for all $t\geq 0$, $k\in\NN$}.
\end{equation}
More precisely, for the original norm the constant $\widetilde M$ can be obtained from \eqref{eq:const_eq_norm} as
$$
\widetilde M = \left(\frac{\max\{1,\tau(0)\}}{\min\{1,\tau(-1)\}}\right)^{\frac{1}{p}}.
$$
Recall that $\tau(-1)$ is (in modulus) the jump of the delay operator at $-1$ (ensuring that we really have a delay equation), and $\tau(0)$ is the total variation of $\eta$.

\medskip
Further, by the calculations in Section \ref{sec:delay_sg} and by the results in B\'{a}tkai and Piazzera \cite[Section~3.1]{Batkai-Piazzera}, we know that the semigroup $\calT$ generated by $(\calG,D(\calG))$ is given by
\begin{equation}\label{eq:delay_sq}
\calT(t)\tbinom{x}{f} = \begin{pmatrix} u(t) \\ u_t \end{pmatrix},
\end{equation}
and we have
\begin{equation}\label{eq:vtvt}
u_t(\sigma)=\begin{cases}
u(t+\sigma)&\mbox{if $t+\sigma>0$},\\
f(t+\sigma)&\mbox{if $t+\sigma\leq 0$}.
\end{cases}
\end{equation}

\medskip Here is now the assumption replacing the range condition. The price we shall pay for weakening the assumption is, however, that we will have to impose more regularity on the initial condition.
\begin{assumption}\label{ass:new1}
\begin{enumerate}[a)]
\item Suppose that $B$ generates a linear contraction semigroup $V$ on $H$.
\item %
Suppose that the linear operator
\begin{equation*}
\Psi:\Ce([-1,0];H)\to H
\end{equation*}
is given by
 $$
 \Psi f=\int\nolimits_{-1}^0 \dd \eta(\sigma) f(\sigma),
 $$
 with $\eta:[1,0]\to \LLL(H)$ of bounded variation, such that for some $-1<\sigma_0<0$
 $$
 \eta:[\sigma_0,0]\to \LLL(H)\quad\mbox{is Lipschitz continuous}.
 $$
 (Note that $\Psi$ can be considered as a bounded operator $\Psi:\W^{1,r}([-1,0];H)\to H$ for each $r\in[1,\infty)$.)
\item Suppose  that for $x\in D(B)$ we have $\eta(s)x\in D(B)$ for all $s\in [-1,0]$. Moreover, suppose that
 $\eta:[-1,0]\to \LLL(D(B))$ is of bounded variation such that for the same $-1<\sigma_0<0$ we have
$$
\eta:[\sigma_0,0]\to \LLL(D(B))\quad\mbox{is Lipschitz continuous}.
$$
\item Suppose further that
\begin{equation*}
    \eta(-1)=0 \text{ and that } \lim_{\sigma\to -1} \eta(\sigma)\neq 0.
    \end{equation*}
\item Suppose  that the function $g:H\to H$ leaves $D(B)$ invariant and is globally Lipschitz both on $H$ and $D(B)$, with Lipschitz constant $\beta$. We set $\Phi=g\circ\Psi$.
\end{enumerate}
\end{assumption}
Informally, the conditions in Assumption \ref{ass:new1} mean that $\Phi$ maps $D(B)$-valued functions into $D(B)$ and that there are no point delays accumulating around zero, only distributed delays.
Note that if the operator $\Phi$ is for example of the form $\Phi = g\circ\delta_{-1}$, then this condition is always satisfied, while the conditions of the previous Theorem \ref{thm:previous}  are not.

In what follows, we abbreviate the norm notation $\|\cdot\|_{\Ell^p([-1,0];D(B))}$ to $\|\cdot\|_{\Ell^p(D(B))}$, and similarly to all function spaces defined on $[-1,0]$ with values in $H$ or $D(B)$.

\begin{proposition}\label{prop:convlinDB}
Let $p_0>1$ and let $p\in (1,p_0)$,  and suppose that the conditions in Assumption \ref{ass:new1} hold.
Define
$$
\calD_p:=\bigl\{\tbinom{x}{f}\in D(B^2)\times \W^{1,p}([-1,0];D(B)), \:f\in \Lip([-1,0];H),\: f(0)=x\bigr\}.
$$
Then for every $T_{m}>0$ there is a constant $C>0$ independent of $p\in (1,p_0)$  such that for every  $\tbinom{x}{f}\in \calD_p$ and
$$
\bigl\|\bigl(\calT_1(\tfrac{t}{n})\calT_2(\tfrac{t}{n})\bigr)^n\tbinom{x}{f}-\calT(t)\tbinom{x}{f}\bigr\|\leq \frac{Ct^{1+1/p}}{n^{1/p}}\bigl(1+\bigl\|\tbinom{x}{f}\bigr\|_{\calD_p} \bigr),
$$
holds for all $t\in [0,T_{m}]$ and $n\in\NN$,
where for $\tbinom xf\in {\calD_p}$ one defines
$$
\bigl\|\tbinom xf\bigr\|_{\calD_p}:=\|x\|_B+\|Bx\|_B+\|f\|_{\W^{1,p}(D(B))}+\|f\|_{\Lip(H)}.
$$
\end{proposition}

\begin{remark}
Let us  fix a time level $T_{m}$, and set $h=\frac {T_{m}} n$ for some $n\in \NN$ and $K={C}{T_{m}}$.
For $j=1,\dots,n$ and $t_j=jh\in [0,T_{m}]$ we have by Proposition \ref{prop:convlinDB}
\begin{align*}
&\bigl\|\bigl(\calT_1(\tfrac{t_j}{j})\calT_2(\tfrac{t_j}{j})\bigr)^j\tbinom{x}{f}-\calT(t_j)\tbinom{x}{f}\bigr\|\leq\frac{Ct_j^{1+1/p}}{j^{1/p}} \bigl(1+\bigl\|\tbinom{x}{f}\bigr\|_{\calD_p} \bigr)\\
&\quad\quad\quad\quad=\frac{Ch^{1/p}t_j^{1+1/p}}{t_j^{1/p}}\bigl(1+\bigl\|\tbinom{x}{f}\bigr\|_{\calD_p} \bigr)\leq K h^{1/p} \bigl(1+\bigl\|\tbinom{x}{f}\bigr\|_{\calD_p} \bigr),
\end{align*}
which leads to
\begin{equation*}
\max_{j=1,\dots,n}
\bigl\|\bigl(\calT_1(h)\calT_2(h)\bigr)^j\tbinom{x}{f}-\calT(jh)\tbinom{x}{f}\bigr\|\leq K h^{1/p} \bigl(1+\bigl\|\tbinom{x}{f}\bigr\|_{\calD_p} \bigr) ,\quad T_{m}=nh.
\end{equation*}
\end{remark}
\begin{theorem}\label{thm:convlinDB}
Suppose that the conditions in Assumption \ref{ass:new1} hold. Then the sequential splitting is of first order. More precisely, let $x\in D(B^2)$ and $f\in \W^{1,p_0}([-1,0];D(B))\cap\Lip([-1,0];H)$ for some $p_0>1$ with $f(0)=x$.
 Then for every $T_{m}>0$ there is a constant $C>0$  such that
$$
\max_{j=1,\dots, n}\|u^{\text{\rm sq}}(jh)-u(jh)\|\leq Ch\bigl(1+\|x\|_B+\|Bx\|_B+\|f\|_{\W^{1,1}(D(B))}+\|f\|_{\Lip(H)}\bigr)
$$
for $T_m=nh$,  and $n\in\NN$ sufficiently large.
\end{theorem}
\begin{proof}
We use that all appearing semigroups are consistent on the $\Ell^p$-scale, and that we have the formulae $u^{\text{\rm sq}}(jh)=\calP_1(\calT_1(h)\calT_2(h))^j\binom xf$ and $u(jh)=\calP_1\calT(jh)\binom xf$. The assertion follows from Proposition \ref{prop:convlinDB} by letting $p\to 1$ and by taking the $p$-independence of the appearing constants into account.
\end{proof}
The proof of  Proposition  \ref{prop:convlinDB} follows the standard route of stability analysis and estimating the local error using the variation of constants formula.
These are carried out in a series of lemmas
where we first show the invariance of the subspace $\calD_p$ under the delay semigroup, and then show the local error estimate for initial values from this set.

In what follows, we shall work with a fixed $p>1$ and use the notation $\calD$ instead of $\calD_p$. The reader will have no difficulty of keeping track of the $p$-independence of the appearing constants. Most notably, the Lipschitz constants of the semigroups (i.e., $M$ and the type $\gamma$) will not depend on $p\in (1,p_0)$ for a fixed $p_0>1$, see Theorem \ref{thm:webb}.)


As before, denote by $\calP_1:\calE\to H$ and $\calP_2:\calE\to \Ell^p([-1,0];H)$ the coordinate projections.

\begin{lemma}\label{lem:lipOK}
Let $\tbinom{x}f\in D(\calG)$ be such that $f\in\Lip([-1,0];H)$ with Lipschitz constant $L$, and let $T_{m}>0$. Then for all $t\in [0,T_{m}]$ the function $u_t:=\calP_2\calT(t)\tbinom{x}f$ is Lipschitz continuous with constant
$$
\max\bigl\{L,M\bigl\|\calG\tbinom{x}f\bigr\|\bigr\},
$$
with $M$ depending only on $\calT$ and $T_{m}$.
\end{lemma}
\begin{proof}
Since $u(s)=\calP_1\calT(s)$, by \eqref{eq:vtvt} and by Miyadera \cite[Theorem 4.2]{miyadera:77} we obtain that
$$
u_t:[\max\{-1,-t\},0]\to H,\quad u_t(\sigma)=u(t+\sigma)
$$
is Lipschitz continuous on $[\max\{-1,-t\},0]$  with constant
$M\|\calG\tbinom{x}f\|$ with $M$ depending only on $\calT$ and $T_{m}$. Now, if $t\in[0,1]$, since $f$ is assumed to be Lipschitz continuous and since $f(0)=x=u_t(-t)$ (for $\tbinom xf\in D(\calG)$), and also by taking into account \eqref{eq:vtvt} again, it follows that $u_t$ is Lipschitz continuous with the asserted constant.
\end{proof}

\begin{lemma}\label{lem:diff_u_t}
Let $\tbinom{x}{f}\in \calE$,  let $v_t:=\calP_2(\calT_2(t)\tbinom{x}{f})$ and  $u_t:=\calP_2(\calT(t)\tbinom{x}{f})$.
Then for all $t\in [0,-\sigma_0]$ we have
$$
\Phi v_t=\Phi u_t+\calO(t),
$$
where $\calO(t)$ denotes a term of magnitude $t\cdot C (\|x\|+\|f\|_{\Ell^p(H)})$ with a constant $C$ independent of $x$ and $f$.
\end{lemma}
\begin{proof}
Similarly to \eqref{eq:vtvt} we have
$$
v_t(\sigma)=\begin{cases}
v(t+\sigma)&\mbox{if $t+\sigma>0$},\\
f(t+\sigma)&\mbox{if $t+\sigma\leq 0$}.
\end{cases}
$$
Using that $g$ is Lipschitz continuous with constant $\beta$, we can write
\begin{align*}
\|\Phi &u_t-\Phi v_t\|\leq \beta \Bigl\|\int\nolimits_{-1}^0 \dd\eta(\sigma) \left(u_t(\sigma)-v_t(\sigma)\right)\Bigr\|\\
&\leq \beta\int\nolimits_{-1}^{-t}\| u_t(\sigma)-v_t(\sigma)\| \cdot\|\dd\eta(\sigma)\|+\beta\int\nolimits_{-t}^{0}\| u(t+\sigma)-v(t+\sigma)\| \cdot\|\dd\eta(\sigma)\|\\
&= \beta\int\nolimits_{-1}^{-t}\hskip-0.4em\| f(t+\sigma)-f(t+\sigma)\|\cdot \|\dd\eta(\sigma)\|+\beta\int\nolimits_{-t}^{0}\hskip-0.4em\| u(t+\sigma)-v(t+\sigma)\| \cdot\|\dd\eta(\sigma)\|\\
&=0+\beta\int\nolimits_{-t}^{0}\| u(t+\sigma)-v(t+\sigma)\| \cdot\|\dd\eta(\sigma)\| \leq \beta\int\nolimits_{-t}^{0}\hskip-0.4em 2M(\|x\|+\|f\|_{\Ell^p(H)})\cdot\|\dd\eta(\sigma)\|
\intertext{where $M$ depends on the bounds of the semigroups $\calT$, $\calT_2$. Since $\eta$ is Lipschitz continuous, we can further estimate}
&\leq 2M\beta \|\eta\|_{\Lip[\sigma_0,0]}(\|x\|+\|f\|_{\Ell^p(H)}) t.\qedhere
\end{align*}
\end{proof}

\begin{lemma}\label{lem:inv_lip}
\begin{enumerate}[a)]
\item The subspace $\calD$ is invariant under the semigroup $\calT$.
\item For $\tbinom{x}{f}\in \calD$ the function
$$
[0,-\sigma_0]\ni s\mapsto \Phi u_s\in H
$$
is Lipschitz continuous with constant $C(1+\|x\|_B+\|f\|_{\Lip(H)}+\|f\|_{\W^{1,p}(H)})$.
\item For $\tbinom{x}{f}\in \calD$ the function
$$
[0,-\sigma_0]\ni s\mapsto \Phi u_s\in D(B)
$$
is bounded in norm, and the function
$$
[0,-\sigma_0]\ni s\mapsto V(s)\Phi u_s\in H
$$
is Lipschitz continuous with constant $C(1+\|x\|_B+\|f\|_{\W^{1,p}(D(B))}+\|f\|_{\Lip(H)}+\|Bx\|_B)$, where the constant $C$ does not depend on $x$ and $f$.
\end{enumerate}
\end{lemma}
\begin{proof}
a)
Let $\widetilde\calT$ be the semigroup generated by the part $\widetilde\calG$ of $\calG$ in $\widetilde\calE := D(B)\times \Ell^p([-1,0];D(B))$. This semigroup exists by the general facts presented in the beginning of Section \ref{sec:delay_sg} and by Assumption \ref{ass:new1}. The domain of the generator is
$$
D(\widetilde\calG)=\bigl\{\tbinom{x}{f}\in D(B^2)\times \W^{1,p}([-1,0];D(B)),\: f(0)=x\bigr\}.
$$
Hence
$$
\calD=D(\widetilde \calG)\cap \left(H\times \Lip([-1,0];H)\right).
$$
The space $D(\widetilde\calG)$ is invariant under  $\widetilde\calT$. It is easy to see that for $\tbinom{x}{f}\in \calD$ we have $\calT(t)\tbinom{x}{f}=\widetilde\calT(t)\tbinom{x}{f}$ for all $t\geq 0$. This follows from the fact that $\calD\subset D(\calG)$,
from the construction of semigroups in terms of the resolvent of their generators, and from the comparison of the topologies of $\calE$ and $\widetilde\calE$.
Next, since $D(\widetilde\calG)\subseteq D(\calG)$, from Lemma \ref{lem:lipOK} we obtain that $\calT(t)\tbinom{x}{f}\in \calD$.

\medskip\noindent b) Let $0\leq s \leq t \leq -\sigma_0$ and take $\tbinom{x}{f}\in\calD$. By the definition of $u_t$ we have
\begin{align*}
\left\| \Phi u_{s} - \Phi u_t \right\| &\leq \beta\left\| \int\nolimits_{-1}^0  \dd\eta(\sigma) \left(u_{s}(\sigma)-u_t(\sigma)\right)\right\| \leq
\beta\int\nolimits_{-1}^0 \| u_{s}(\sigma) - u_t(\sigma)\|\cdot \| \dd \eta(\sigma)\| \\
&=\beta\int\nolimits_{-1}^{-t} \| f({s}+\sigma)-f(t+\sigma)\| \cdot \| \dd \eta(\sigma)\| \\
&\qquad + \beta\int\nolimits_{-t}^{-{s}} \| f({s}+\sigma) - u(t+\sigma)\| \cdot \| \dd \eta(\sigma)\| \\
&\qquad+ \beta\int\nolimits_{-{s}}^0 \| u({s}+\sigma)-u(t+\sigma)\| \cdot \| \dd \eta(\sigma)\|.
\end{align*}
The first of these three terms can be estimated as
\begin{equation*}
\beta\int\nolimits_{-1}^{-t} \| f({s}+\sigma)-f(t+\sigma)\| \cdot \| \dd \eta(\sigma)\| \leq  \beta\|f\|_{\Lip(H)} \|\eta\|_{\BV(\LLL(H))}\cdot (t-{s}).
\end{equation*}

Similarly, since $u$ is a mild solution of the delay equation (by $\calD\subset D(\widetilde\calG)\subset D(\calG)$), by Miyadera \cite[Theorem 4.2]{miyadera:77} it is Lipschitz continuous with Lipschitz constant  $\|\calG\tbinom{x}{f}\|$. Hence the third term can be bounded in a similar way as the first: For the Lipschitz constant $L$ of the integrand we obtain
\begin{equation*}
L= \|\calG\tbinom{x}{f}\|\leq \bigl(\|Bx+\Phi f\|+\|f\|_{\W^{1,p}(H)}\bigr).
\end{equation*}
Finally, in the second term we integrate on the interval $[-t,-{s}]$ with $t,{s}\leq -\sigma_0$. Hence this term can be bounded using the Sobolev inequality by $C(t-{s})(1+\|x\|+\|f\|_{\Ell^\infty(H)})$, where $C$ depends  on $\calT$ and on the Lipschitz constant of $\eta$ on $[\sigma_0,0]$.

\medskip\noindent c)  By the considerations in part a), we have that for $\tbinom{x}{f}\in \calD$, the function
 $$ s\mapsto \calP_2\widetilde \calT(s)\tbinom{x}{f}\in \W^{1,p}([-1,0];D(B))$$ is continuous, and hence especially bounded on $[0,1]$ by
  $$
  \sup_{s\in[0,1]}\|\calP_2\widetilde \calT(s)\tbinom{x}{f}\|_{\W^{1,p}(D(B))}\leq M\bigl(1+\|Bx\|_B+\|f\|_{\W^{1,p}(D(B))}\bigr).
  $$
  This shows the first assertion, since $\Phi$ maps bounded subsets of $\W^{1,p}([-1,0];D(B))$ to bounded sets in $D(B)$ by Assumption \ref{ass:new1}. The second statement is a trivial consequence now, since for large $\lambda>0$ the function  $[0,1]\ni s\mapsto V(s)R(\lambda, B)\in \LLL(H)$ is Lipschitz, hence so is
$s\mapsto V(s)R(\lambda,B)(\lambda-B)\Phi u_s$ if we take into account also part b). The Lipschitz constant is
$$
L=C\bigl(1+\|x\|_B+\|f\|_{\W^{1,p}(D(B))}+\|f\|_{\Lip(H)}+\|Bx+\Phi f\|_B\bigr),
$$
where $C$ does not depend on $x$ and $f$.
\end{proof}

\begin{lemma}[{\bf Local error\rm}]\label{lem:loc_error}
There is a constant $K>0$ such that for all $t\in [0,-\sigma_0)$ we have
$$
\bigl\|\calT_1(t)\calT_2(t)\tbinom{x}{f}-\calT(t)\tbinom{x}{f}\bigr\| \leq Kt^{1+\frac1p} \bigl(1+\bigl\|\tbinom xf\bigr\|_\calD\bigr)\quad\mbox{for all $\tbinom xf\in \calD$}.
$$
\end{lemma}
\begin{proof}
 Take $\tbinom{x}{f}\in \calD$, and let $u$, $u_t$, $v$, $v_t$ be as in \eqref{eq:split_sgs} and \eqref{eq:delay_sq}, respectively. Using the variation constants formula for $u$, which is valid by \cite[Proposition 5.8]{Webb76} since $p>1$,  by the linearity of $V(t)$ (see Assumption \ref{ass:new1}) we obtain
\begin{align*}
&V(t)v(t)-u(t) = V(t)x + V(t)\int\nolimits_0^t \Phi v_s \dd s - V(t)x - \int\nolimits_0^t V(t-s)\Phi u_s \dd s.
\intertext{By Lemma \ref{lem:diff_u_t}, this is further equal to}
& V(t)\int\nolimits_0^t \Phi u_s \dd s - \int\nolimits_0^t V(t-s)\Phi u_s\dd s  +\int\nolimits_0^t \calO(s)\dd s \\
&\quad= \int\nolimits_0^t V(t-s)\left( V(s)\Phi u_s-\Phi u_s\right) \dd s+ \calO(t^2)\\
&\quad=\int\nolimits_0^t V(t-s)\left( V(s)\Phi u_s-\Phi u_0\right) \dd s+\int\nolimits_0^t V(t-s)\left(\Phi u_0-\Phi u_s\right) \dd s
 +\calO(t^2).
\end{align*}
By Lemma \ref{lem:inv_lip}, the function $[0,-\sigma_0]\ni s\mapsto \Phi u_s$ is Lipschitz continuous, and the function
$
[0,-\sigma_0]\ni s\mapsto V(s)\Phi u_s$
is also Lipschitz continuous. Denoting by $L$ the larger of their Lipschitz constants, we obtain that
\begin{align*}
\|V(t)v(t)-u(t)\| \leq
M L t^2 +MLt^2+\|\calO(t^2)\| \leq C_0t^2 \bigl(1+\bigl\|\tbinom{x}{f}\bigr\|_{\calD}\bigr).
\end{align*}
Note that we also used Lemma \ref{lem:diff_u_t} again.

As for the second coordinates and $t\in [0,-\sigma_0)$, we can make the following computations:
\begin{align*}
v(t)-u(t)&=x+\int_0^t \Phi v_r\dd r-V(t)x-\int_0^t V(t-r)\Phi u_r\dd r\\
&=x-V(t)x+\int_0^t \calO(r)\dd r+\int_0^t (I-V(t-r))\Phi u_r\dd r\\
&=x-V(t)x+\calO(t^2)+\int_0^t (I-V(t-r))\Phi u_r\dd r,
\end{align*}
where  $\calO(t^2)$ denotes a term of order of magnitude $C_1t^2(1+\|x\|+\|f\|_{\Ell^p(H)})$, see Lemma \ref{lem:diff_u_t}. We obtain from Lemma \ref{lem:inv_lip}.c) that
\begin{align*}
&\Bigl\|\int_0^t (I-V(t-r))\Phi u_r\dd r\Bigr\|\leq \int_0^t \|(I-V(t-r))\Phi u_r\|\dd r\\
&\quad\leq \int_0^t (t-r) \dd r\cdot C_2(1+\|x\|_B+\|f\|_{\W^{1,p}(D(B))}+\|f\|_{\Lip(H)}+\|Bx\|_B)=\calO(t^2)
\end{align*}
with implied constant $C_3\|\tbinom xf\|_{\calD}$. Since $x\in D(B)$ we have $x-V(t)x=C_4t\|Bx\|$.
By putting together these estimates we obtain
\begin{align*}
\|u_t-v_t\|^p_{\Ell^p(H)}&=\int_{-t}^0 \|u(t+\sigma)-v(t+\sigma)\|^p\dd \sigma\leq C_5 t^{p+1}\bigl(1+\|\tbinom xf\|_\calD\bigr).
\end{align*}
This finishes the proof by defining $K$ in terms of the previously appearing constants $C_i$.
\end{proof}

\begin{proof}[Proof of Proposition \ref{prop:convlinDB}]
Without loss of generality we may assume $T_{m}=1$ (the constant $C$ shall depend on $T_{m}$, however). By a standard telescopic summation we obtain
\begin{align}\label{eq:tele}
&\bigl(\calT_1(\tfrac{t}{n})\calT_2(\tfrac{t}{n})\bigr)^n-\calT(t) \\
&\qquad = \sum_{k=0}^{n-1}\bigl((\calT_1(\tfrac{t}{n})\calT_2(\tfrac{t}{n})\bigr)^{n-k}  \calT_1(\tfrac{t}{n})\calT_2(\tfrac{t}{n})\calT(\tfrac{kt}{n}) \notag\\
&\qquad \qquad \qquad \qquad \qquad - (\calT_1(\tfrac{t}{n})\calT_2(\tfrac{t}{n})\bigr)^{n-k}\calT(\tfrac{t}{n})\calT(\tfrac{kt}{n})\bigr).\notag
\end{align}
As we saw in Lemma \ref{lem:inv_lip}, the set $\calD$ is invariant under the delay semigroup $\calT$. Also from this lemma  and from  Lemma \ref{lem:lipOK}, it follows that for some $M\geq 1$
\begin{equation*}
\bigl\|\calT(s)\tbinom xf\bigr\|_{\calD} \leq M\bigl(1+\bigl\|\tbinom xf\bigr\|_{\calD}\bigr)\quad \mbox{for all $s\in [0,1]$, $\tbinom xf\in \calD$.}
\end{equation*}
By \eqref{eq:stab} we can bound the Lipschitz norm of the first factors on the right-hand side of \eqref{eq:tele} by exponential terms. Hence, by applying Lemma \ref{lem:loc_error} we obtain that
\begin{align*}
&\left\|\left(\bigl(\calT_1(\tfrac{t}{n})\calT_2(\tfrac{t}{n})\bigr)^n-\calT(t)\right)\tbinom{x}{f}\right\| \leq
\sum_{k=0}^{n-1} \widetilde M\ee^{2\gamma kt/n} K \left(\tfrac{t}{n}\right)^{1+1/p} \bigl(1+\bigl\| \calT(\tfrac{kt}{n})\tbinom{x}{f}\bigr\|_{\calD}\bigr)\\
&\qquad\leq\frac{Ct^{1+1/p}}{n^{1/p}}\bigl(1+\bigl\| \tbinom{x}{f}\bigr\|_{\calD}\bigr),
\end{align*}
with an appropriate constant $C\geq 0$ (depending on $M$, $K$, $\widetilde M$, $\gamma$ and $T_{m}$).
\end{proof}



\section{Numerical illustration}\label{sec:num}

\noindent In the following we illustrate the results on a delayed diffusion problem which has been chosen to illustrate the efficiency and applicability of operator splitting procedures to partial differential equations with delay.
The aim of this section is to demonstrate
\begin{itemize}
\item that the theoretical calculations are sharp in the sense that they can also be practically achieved, and
\item that the splitting does not introduce additional large error constants.
\end{itemize}
These goals will be achieved by calculating the order of convergence and by comparing the method to the implicit Euler scheme, the standard first order method to solve such kind of equations.

Consider the problem
\begin{equation}\label{eq:numexample}
\left\{
\begin{aligned}
\tfrac\partial{\partial t} w(t,x)&=\tfrac{\partial^2}{\partial x^2} w(t,x)+\sin\Big(w(t-1,x)+\int\nolimits_{-1}^{-0.5}\sigma\cdot w(t+\sigma,x)\dd\sigma\Big), t\geq 0, \\
w(t,0)&=w(t,\pi)=0, \quad x\in[0,\pi], \\
w(\sigma,x)&=\ee^{\sigma}\cdot x(\pi-x), \quad x\in [0,\pi],\: \sigma\in[-1,0].
\end{aligned}
\right.
\end{equation}
We split the equation into the undelayed diffusion equation
$$
\begin{cases}
\tfrac\partial{\partial t} w(t,x)=\tfrac{\partial^2}{\partial x^2} w(t,x),\\
w(t,0)=w(t,\pi)=0
\end{cases}
$$
subject to homogeneous Dirichlet boundary condition, and the pure delayed problem
$$
\tfrac\partial{\partial t} w(t,x)=\sin\Big(w(t-1,x)+\int\nolimits_{-1}^{-0.5}\sigma\cdot w(t+\sigma,x)\dd\sigma\Big)
$$
with the initial history function
$$
f(\sigma,x) = \ee^{\sigma}\cdot x(\pi-x) \quad\text{for}\quad \sigma\in[-1,0].
$$
We solve the heat equation by Fourier method via discrete (fast) Fourier transform. The delay problem is solved by successive integration using trapezoidal rule. We define the time-dependent vector $\mathbf{w}^n$ with elements $\mathbf{w}_j^n=w(nh,j\dx)$, $j=0,\dots,N=\frac{\pi}{\dx}$, $n\in\mathbb{N}$, where $h$ and $\dx$ denotes the time step and  the spatial grid size, respectively. We then obtain the formulae
\begin{align*}
&\boldsymbol\Lambda^{n+1} = \tfrac{h}{2}\sum_{i=0}^{\frac{0.5}{h}}\left((-1+ih)\cdot\mathbf{w}^{n+1-k+i}+(-1+(i+1)h)\cdot\mathbf{w}^{n+1-k+i+1}\right), \\
&\mathbf{w}^{n+1} = \mathbf{w}^n+h\cdot\sin\left(\mathbf{w}^{n+1-k}+\boldsymbol\Lambda^{n+1}\right),\quad\text{with}\quad h=\tfrac{1}{k},\, k\in\mathbb{N}.
\end{align*}
Note that the identity $\mathbf{w}_j^l = f(-1+(n+l)h,j\dx)$ holds for $l\le 0$.
%
Since both limits of the integral are negative, the history function is evaluated only in the past if the time step $h\le0.5$. To this end the numerical solution is saved for the time levels $n-k,\dots ,n$. The results are presented for $N=128$ grid points in the interval $[0,\pi]$. We obtained similar results also for other values of $N$.

Next we compare our numerical results with those obtained by using the standard first order method, namely the implicit Euler method. In this case no splitting procedure was applied and the Laplacian was approximated by the second-order finite difference method. Then the boundary condition means $\mathbf{w}_0^{n+1}=\mathbf{w}_{\pi/\dx}^{n+1}=0$, and for $j=1,\dots ,\tfrac{\pi}{\dx}-1$ one has
\begin{equation*}
\left(I-\tfrac{h}{\dx^2}\mathrm{tridiag}(1,-2,1)\right)\mathbf{w}^{n+1} = \mathbf{w}^n+h\cdot\sin\left(\mathbf{w}^{n+1-k}+\boldsymbol\Lambda^{n+1}\right)
\end{equation*}
with the notation above. We have again $\mathbf{w}_j^l = f(-1+(n+l)h,j\dx)$ for $l\le 0$.

The reference solution, which is needed for the error analysis, was also computed by the implicit Euler method but with finer spatial and temporal resolutions ($N=256$ grid points and time step $h=10^{-4}$). The codes are written in \textsc{matlab}. The system of linear equations, which appears when using the implicit Euler method, is solved by the ``backslash'' command, and its tridiagonal matrix is stored as a sparse matrix.

We examined the behavior of the relative global error and the CPU time for various values of the time step for time $t=2$. The numerical experiment shows that the splitting is of first order as analytically proved in Theorem \ref{thm:convlinDB}. In Figure \ref{fig:order} the resulting global errors are shown for various values of the time step $h$. Fitting a line on the results displayed in a logarithmic scale yields the numerical approximation to the order, which approximately equals 1 also for the splitting method.

\begin{figure}[!ht]
	\begin{center}
{\includegraphics[width=10cm]{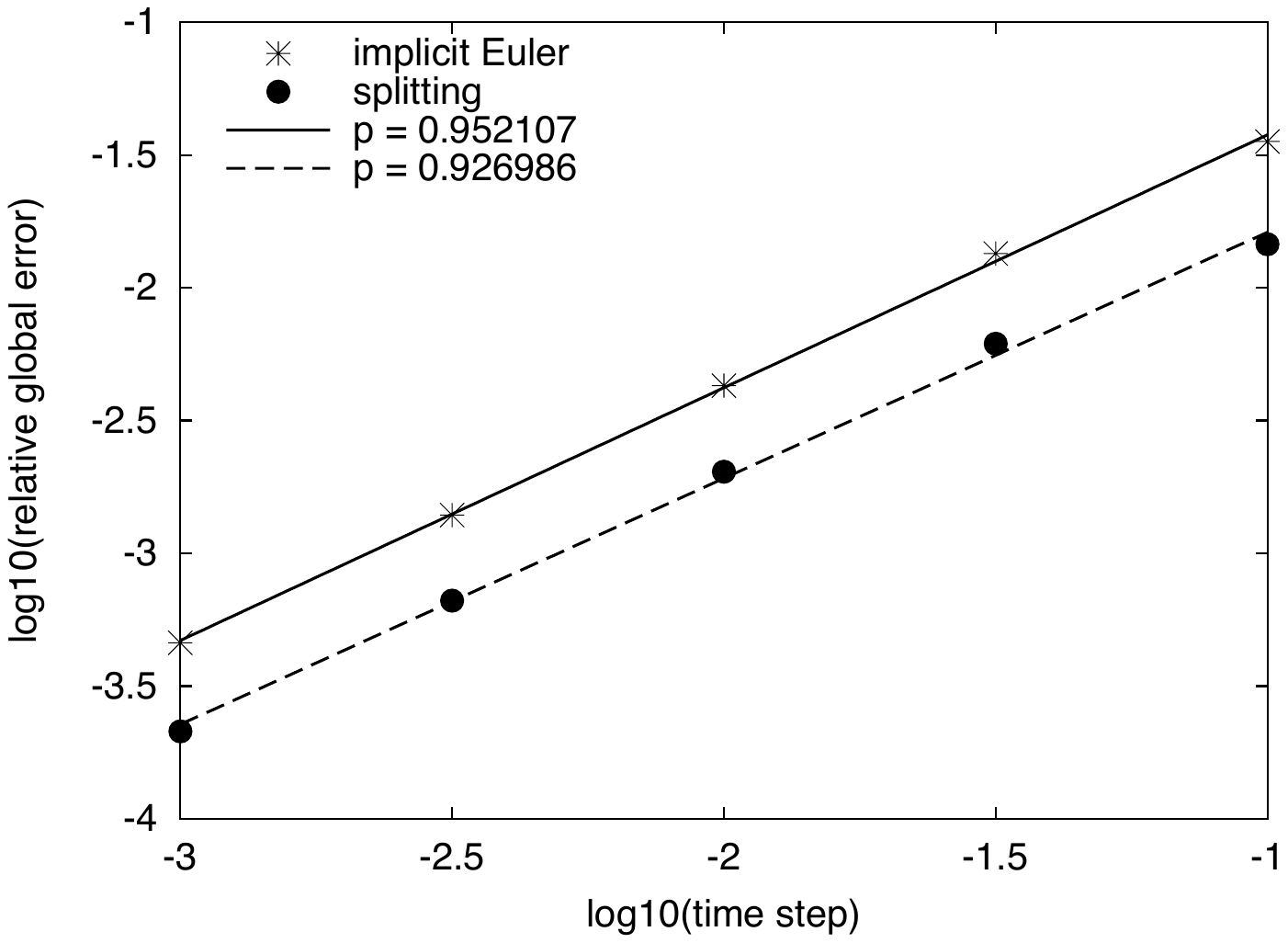}}
\end{center}
\caption{Order $p$ determined numerically at time level $t=2$ (for $N=128$ grid points). \label{fig:order}}
\end{figure}

We note that up to time $t=2.5$ the splitting method yields more accurate solutions than the implicit Euler method. After that time the implicit Euler method becomes more accurate. In Figure \ref{fig:cpu} the computational (CPU) time of both methods are presented for various values of the time step at time $t=2$. One can clearly see that up to a certain value of the time step the splitting method performs faster than the implicit Euler method.

\begin{figure}[!ht]
\begin{center}
{\includegraphics[width=10cm]{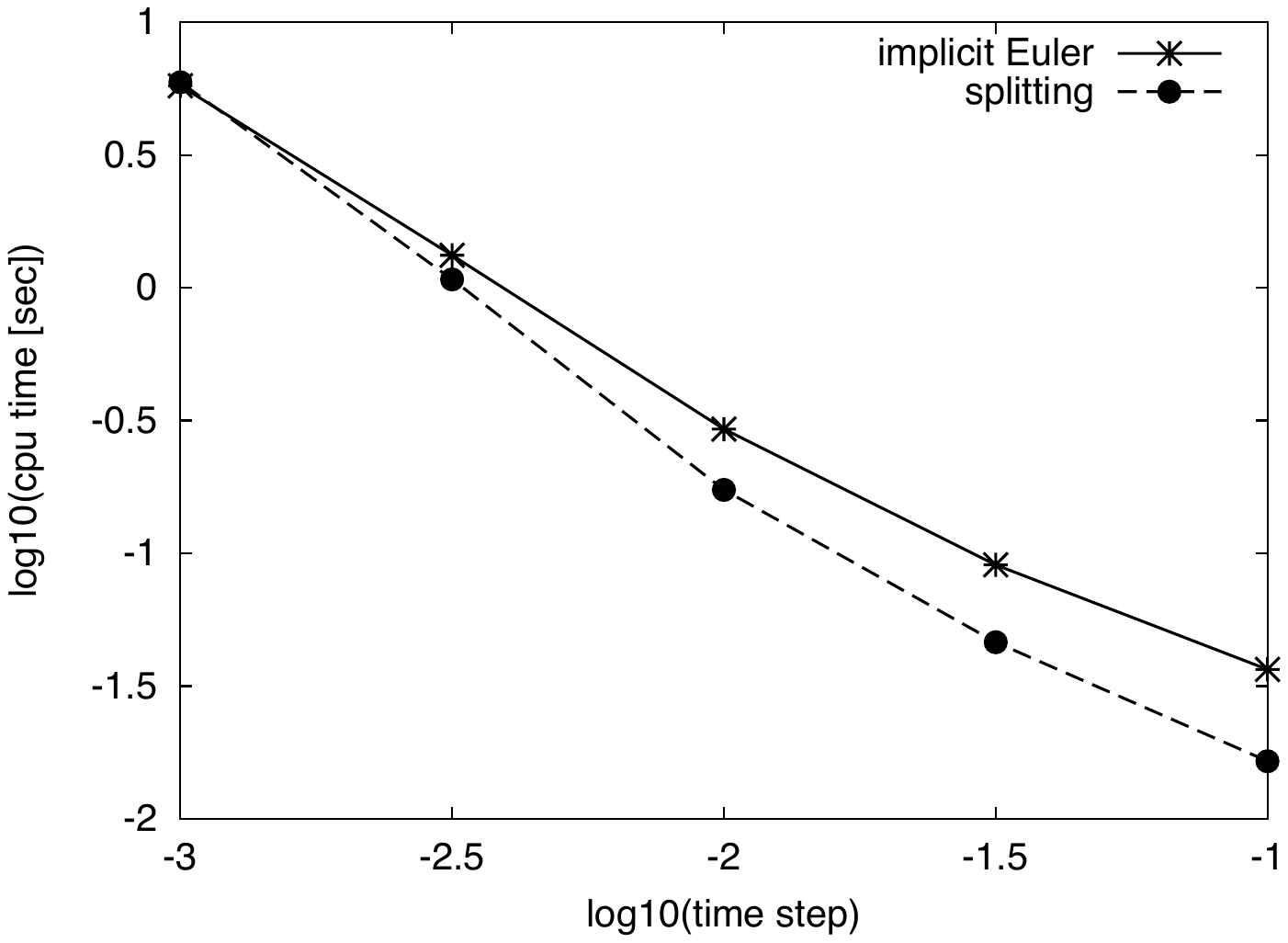}}
\end{center}
\caption{CPU times at the time level $t=2$ (for $N=128$ grid points). \label{fig:cpu}}
\end{figure}

Figure \ref{fig:cpuvserror} shows the efficiency of both methods. For this example the splitting outperforms  the implicit Euler method.

\begin{figure}[!ht]
\begin{center}{\includegraphics[width=10cm]{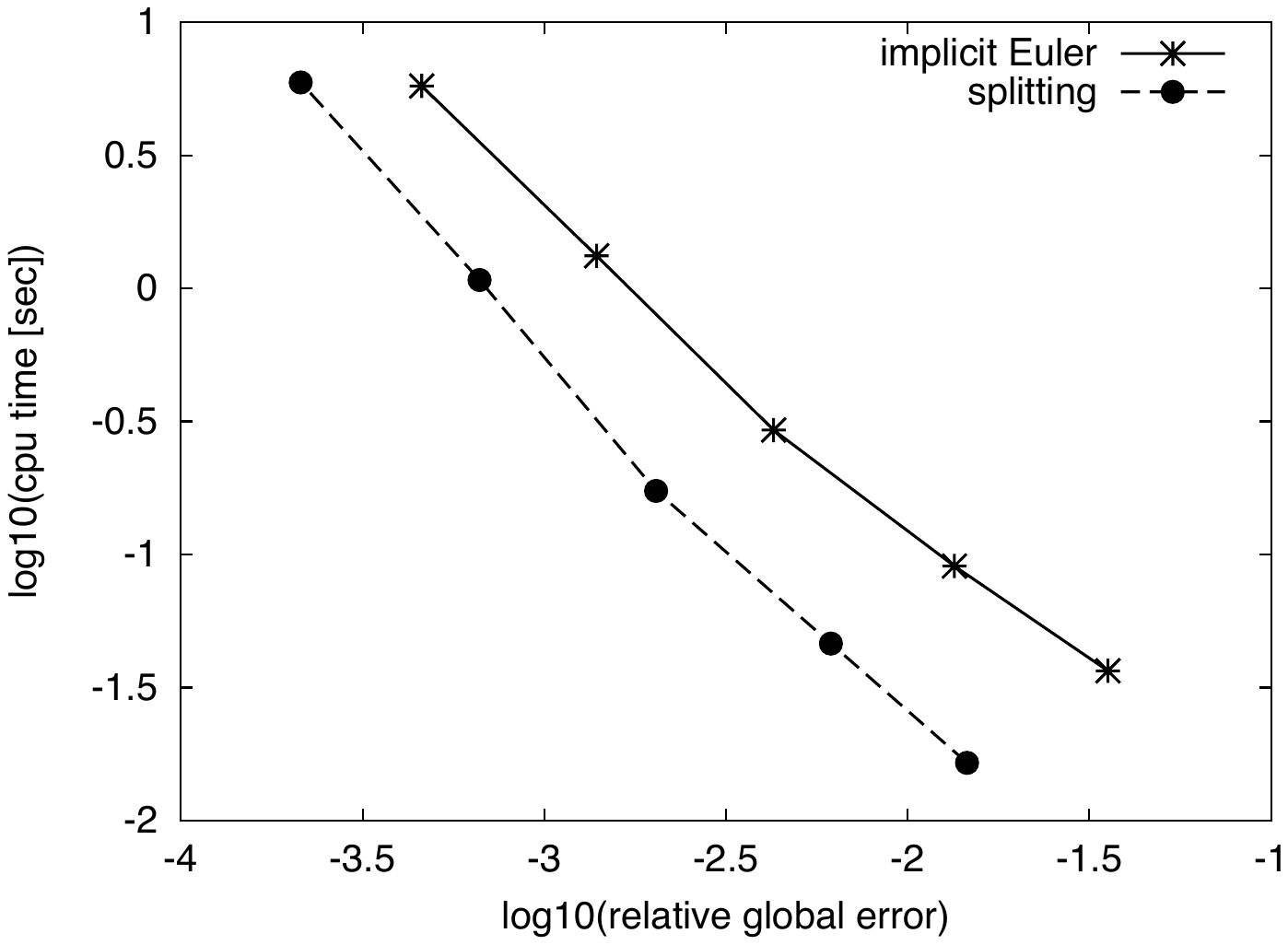}}\end{center}
\caption{CPU times against the error at the time level $t=2$ (for $N=128$ grid points). \label{fig:cpuvserror}}
\end{figure}\
The reason of this case study was to show that the application of splitting offers an alternative method for solving delay equations which can
be naturally split into more sub-problems. The split sub-problems can be solved by applying the appropriate method for each of them (e.g., the Fourier
method for the Laplacian in our case). One cannot forget about the splitting error of course, which we proved to be of first order, and numerical
experiments illustrated the theoretical results. With the help of the case study above, however, we have demonstrated that splitting methods perform
comparably to well-established numerical methods (of the same order), when the structure of the equation can be exploited.

\section*{Acknowledgments}
A.~B.~was supported by a DAAD guest professorship at the university of Wuppertal.


\parindent0pt

\end{document}